\documentclass{amsart}
\usepackage{amsmath,amssymb,amsthm,eucal}
\newcommand{\ds}{\displaystyle}

\begin{document}

\title{Almost Everywhere Convergence Of Convolution Measures}
\author{Karin Reinhold \\ Anna K. Savvopoulou \\  Christopher M. Wedrychowicz} 
\address{Department of Mathematics\\
         University at Albany, SUNY\\ 
         Albany, NY 12222}
         \address{Department of Mathematical Sciences \\
         Indiana University in South Bend, \\
         South Bend, IN, 46545}
         \address{Department of Mathematical Sciences \\
         Indiana University in South Bend, \\
         South Bend, IN, 46545}
         
\email{reinhold@albany.edu \\ annsavvo@iusb.edu \\ cwedrych@iusb.edu}
\maketitle

\begin{abstract}
Let $(X,\mathcal{B},m,\tau)$ be a dynamical system with $\ds (X,\mathcal{B},m)$ a probability space and $\ds \tau$ an invertible, measure preserving transformation. The present paper deals with the almost everywhere convergence in $\ds\mbox{L}^1(X)$ of a sequence of operators of weighted averages. Almost everywhere convergence follows once we obtain an appropriate maximal estimate and once we provide a dense class where convergence holds almost everywhere. The weights are given by convolution products of members of a sequence of probability measures $\ds\{\nu_i\}$ defined on $\ds\mathbb{Z}$. We then exhibit cases of such averages, where convergence fails. 
\end{abstract}

\section{Introduction}


\subsection{Preliminaries}
Let $\ds (X,\mathcal{B},m)$ be a non-atomic, separable probability space. Let $\ds\tau$ be an invertible, measure preserving
transformation of $\ds (X,\mathcal{B},m)$. Given a probability measure $\ds\mu$ defined on $\ds\mathbb{Z}$ , one can define
the operator $\ds \mu f(x)=\sum_{k\in\mathbb{Z}}\mu(k)f(\tau^k x)$ for $\ds x\in X$ and $\ds f\in \mbox{L}^p(X)$ where $\ds p\geq 1$.
Note that this operator is well defined for almost every $x\in X$ and that it is a positive contraction in all $\ds \mbox{L}^p(X)$ for $\ds p\geq 1$,
i.e $\ds \|\mu f\|_{p}\leq \|f\|_p$. 
 
Given a sequence of probability measures, $\ds\{\mu_n\}$, defined on $\ds\mathbb{Z}$, one can subsequently define a sequence of operators
as follows, $\ds\mu_n f(x)=\sum_{k\in\mathbb{Z}}\mu_n(k)f(\tau^k x)$. The case where the weights are induced by the convolution powers
of a single probability measure defined on $\ds\mathbb{Z}$ has already been studied. More specifically, given $\ds\mu$ a probability
measure on $\ds\mathbb{Z}$, let $\ds\mu^n$ denote the $\ds n^{\mbox{th}}$ convolution power of $\ds\mu$, which 
is defined inductively as $\ds \mu^n=\mu^{n-1}\ast\mu$ where $\ds \mu^2(k)=(\mu\ast\mu)(k)=\sum_{j\in\mathbb{Z}}\mu(k-j)\mu(j)$ for all $\ds k\in\mathbb{Z}$. In ~\cite{cald} and 
~\cite{belrjblatt} the authors studied the sufficient conditions on $\ds\mu$ that give $\ds \mbox{L}^p$, ($p\geq 1$), convergence 
of the sequence of operators of the form 
$$\ds \mu_n f(x)=\sum_{k\in\mathbb{Z}}\mu^n(k)f(\tau^k x).$$
 
The type of weighted averages that will be considered in this paper, are averages whose weights are induced by the convolution product of members of a sequence of probability measures $\ds\{\nu_i\}$ defined on $\ds\mathbb{Z}$. Given this sequence of probability measures $\ds\{\nu_i\}$,
we define another sequence of probability measures $\ds\{\mu_n\}$ on $\ds\mathbb{Z}$ in the following way,
\begin{eqnarray*}
\mu_1 & = & \nu_1\\
\mu_2 & = & \nu_1\ast\nu_2 \\
& \vdots & \\
\mu_n & = & \nu_1\ast\cdots\ast\nu_n
\end{eqnarray*}

We then define the following sequence of operators
$$\mu_n f(x)=\sum_{k\in\mathbb{Z}}(\nu_1\ast\cdots\ast\nu_n)(k)f(\tau^k x)=\sum_{k\in\mathbb{Z}}\mu_n(k)f(\tau^k x).$$

Note that these operators $\ds \mu_nf(x)$ are well defined for almost every $\ds x\in X$ and that they are positive contractions in all $\ds\mbox{L}^p(X)$, for $\ds 1\leq p\leq\infty$.

If one defines $\ds T_mf(x)=\sum_{k\in\mathbb{Z}}\nu_m(k)f(\tau^kx)$ we may view $\ds\mu_nf(x)=\nu_1\ast\cdots\nu_n f(x)$ as the composition of $T_1,T_2,\ldots,T_n$ i.e $\ds \mu_nf(x)=T_n\cdots T_1f(x)$. Therefore the almost everywhere convergence of $\ds\mu_nf(x)$ may be viewed as a special case of the almost everywhere convergence of the sequence $\ds S_nf(x)=T_n\cdots T_1f(x)$ where the $\ds T_i$'s are positive contractions of $\ds L^p$ $\ds\forall p\geq 1$. If one defines $\ds S_nf(x)=T_1^{\ast}\cdots T_n^{\ast} T_n\cdots T_1 f(x)$, where $\ds T_i^{\ast}$ denotes the adjoint of $T_i$ one encounters a much studied situation. In our case this would correspond to successive convolution
of $\nu_i$ and $\ds \tilde{\nu}_i$ where $\tilde{\nu}_i$ is defined by $\ds \tilde{\nu}_i(k)=\nu_i(-k)$. When $\ds f\in L^p$ for $\ds 1<p<\infty$ and the $T_i$'s are positive contractions and $\ds T_n1=T_n^{\ast}1=1$ Rota in \cite{rota} established the almost everywhere convergence. In \cite{akcoglu} Akcoglu extended this result to the situation where the $\ds T_i$'s are not necessarily positive. Concerning
$p=1$ Ornstein (\cite{ornstein}) constructed an example of a self-adjoint operator $T$ 
satisfying the above for which $\ds T\cdots Tf(x)=T^nf(x)$ fails to converge
almost everywhere.

The above failure when $p=1$ is in contrast to the almost everywhere convergence of the Cesaro averages $\ds \frac{1}{n}\sum_{k=1}^n T^kf(x)$, (see \cite{peter})


 
 \subsection{Definitions and Past Results}
Before we mention a few of the results regarding weighted averages with convolution powers, some 
definitions are essential.  
 \theoremstyle{definition} 
 \newtheorem{strap}{Definition}[subsection]
 \begin{strap}
 A probability measure $\ds\mu$ defined on a group $G$ is called \textbf{strictly aperiodic} if and only if 
 the support of $\ds\mu$ can not be contained in a proper left coset of $\ds G$.
 \end{strap}
 
\noindent A key theorem by Foguel that we will use repeatedly is the following,

\theoremstyle{plain}
 
\newtheorem{foguel}[strap]{Theorem}
\begin{foguel}[\cite{caldzyg}]
If $\ds G$ is an abelian group and $\ds\hat{G}$ denotes the character group of the group $G$, then the following are equivalent for a probability measure $\ds\mu$,
\begin{enumerate}
\item $\ds\mu$ is strictly aperiodic
\item If $\ds\gamma\neq 1$, $\ds\gamma\in \hat{G}$, then $\ds|\hat{\mu}(\gamma)|<1$ 
\end{enumerate}
\end{foguel}

\theoremstyle{plain}



\noindent A few more definitions,


\theoremstyle{definition}
\newtheorem{expmom}[strap]{Definition}
\begin{expmom}
If $\ds p>0$, the \textbf{$\ds \mathbf{p^{\mbox{th}}}$ moment} of $\mu$ is given by $\ds\sum_{k\in\mathbb{Z}}|k|^p\mu(k)$ and it is
denoted by $\ds m_p(\mu)$. The \textbf{expectation} of $\mu$ is $\ds\sum_{k\in\mathbb{Z}}k\mu(k)$ and is denoted by
$\ds E(\mu)$. 
\end{expmom}


\theoremstyle{plain}




\noindent In \cite{cald} Bellow and Calder\'on proved,

\newtheorem{caldbelweak}[strap]{Theorem}
\begin{caldbelweak}
Let $\ds\mu$ be a strictly aperiodic probability measure defined on $\ds\mathbb{Z}$
that has expectation $0$ and finite second moment. The sequence of operators 
$$\mu_n f(x)=\sum_{k\in\mathbb{Z}}\mu^n(k)f(\tau^k x)$$ converges almost 
everywhere for $\ds f\in\mbox{L}^1(X)$.
\end{caldbelweak}

The proof of the above theorem involves translating properties of the measure into equivalent conditions on the Fourier transform of the measure. 



\section{Convolution measures}

In this section we discuss sufficient conditions on the sequence of probability measures $\ds\{\nu_i\}$ so that
the operators $\ds\mu_n f(x)=\sum_{k\in\mathbb{Z}}\mu_n(k)f(\tau^k x)=\sum_{k\in\mathbb{Z}}(\nu_1\ast\cdots\ast\nu_n)(k)f(\tau^k x)$ converge 
a.e for $\ds f\in\mbox{L}^1(X)$. We will show that the maximal operator of this sequence is of weak-type $(1,1)$ and then we establish a dense class where a.e convergence holds. 
Almost everywhere convergence will follow from Banach's Principle.

\subsection{Maximal Inequality}

\noindent To establish a maximal inequality we will use the following theorems.
 
\newtheorem{thm2}{Theorem}[subsection]
\begin{thm2}[\cite{cald}]
\label{bellowweak}
Let $\ds (\mu_n)$ be a sequence of probability measures on $\mathbb{Z}$, \\
$\ds f:X\rightarrow \mathbb{R}$ and the operators \\
\begin{eqnarray*}
(\mu_nf)(x)=\sum_{k\in\mathbb{Z}}\mu_n(k)f(\tau^kx) 
\end{eqnarray*}
Let $\ds Mf(x)=\sup_{n}|\mu_nf(x)|$, denote the maximal operator.\\
Assume that there is $\ds 0<\alpha\leq 1$ and $\ds C>0$ such that for $\ds n\geq 1$

\begin{eqnarray*}
|\mu_n(x+y)-\mu_n(x)|\leq C\frac{|y|^{\alpha}}{|x|^{1+\alpha}} \:\mbox{for}\:x,y\in\mathbb{Z}\:,\, 2|y|\leq |x|
\end{eqnarray*}
then the maximal operator $\ds M$ satisfies a weak-type (1,1) inequality, namely
there exists $C$ such that for any $\ds\lambda>0$
$$m\left\{x\in X: (Mf)(x)>\lambda\right\}\leq \frac{C}{\lambda}\|f\|_1\:\mbox{for all}\: f\in\mbox{L}^1(X).$$
\end{thm2}





\noindent A sufficient condition to obtain the assumption of Theorem~\ref{bellowweak} is given by the following,

\newtheorem{cor}[thm2]{Corollary}
\begin{cor}[\cite{cald}]
Let $\ds \mu_n$ be a sequence of probability measures defined on $\ds\mathbb{Z}$ and let $\ds \hat{\mu}_n(t)$ denote its Fourier transform
for $\ds t\in[-1/2,1/2)$. We assume that\\
\begin{eqnarray*}
\sup_n\int_{-1/2}^{1/2}|\hat{\mu}''_n(t)||t|\,\mbox{d}t<\infty
\end{eqnarray*}
then there exist $\ds 0<\alpha\leq 1$ and $\ds C>0$ such that for $\ds n\geq 1$

\begin{eqnarray*}
|\mu_n(x+y)-\mu_n(x)|\leq C\frac{|y|^{\alpha}}{|x|^{1+\alpha}} \:\mbox{for}\:x,y\in\mathbb{Z}\:,\, 2|y|\leq |x|\: .
\end{eqnarray*}
\end{cor}

\newtheorem{th3}[thm2]{Theorem}
\begin{th3}
\label{pos}
Let $\ds (\nu_{n})$ be a sequence of strictly aperiodic probability measures on $\ds\mathbb{Z}$ such that
\begin{enumerate}
\item $\ds E(\nu_n)=0$\, , \, $\ds\forall n$
\item
$\ds \phi(n)=\sum_{i=1}^{n} m_2(\nu_i)=O(n) $
\item 
There exist a constant $C$ and an integer $\ds N_0>0$, such that $\ds |\hat{\nu}_{n}(t)|\leq e^{-Ct^2}$, $\ds\forall n>N_0$ and $\ds t\in[-1/2,1/2)$ 
\end{enumerate}
then for $\ds \mu_n=\nu_1\ast\cdots\nu_n$ we have that
$$\sup_{n}\int_{-1/2}^{1/2}|\hat{\mu}''_n(t)||t|\,\mbox{d}t<\infty$$
and therefore the maximal operator
$\ds Mf(x)=\sup_{n\in\mathbb{Z}}|\mu_nf(x)|$ is weak-type $(1,1)$.
\end{th3}
\begin{proof}
Without loss of generality we can assume that $\ds N_0=1$ .\\
Let\\
\begin{eqnarray*} 
a_n=4\pi^2m_2(\nu_{n})
\end{eqnarray*}
Under our hypothesis one can show that for $\ds \hat{\nu}_n(t)=\sum_k \nu_n(k)e^{2\pi ikt}$
and \\$\ds t\in[-1/2,1/2)$,
\begin{eqnarray*}
|\hat{\nu}_{n}'(t)|& \leq & a_{n}|t|,\quad \mbox{for}\:t\in[-1/2,1/2), \\
|\hat{\nu}_{n}''(t)| & \leq & a_n ,\quad \mbox{for}\:t\in[-1/2,1/2).
\end{eqnarray*}
Observe that since $\ds \mu_{n}  = \nu_{1}\ast\cdots\ast\nu_{n}$,
\begin{eqnarray*} 
\hat{\mu}_{n}(t) & = & \prod_{i=1}^{n}\hat{\nu}_i(t), \\
\hat{\mu}'_{n}(t) & = & \sum_{j=1}^{n}\prod_{\ds i=1\atop \ds i\neq j}^{n}\hat{\nu_i}(t)\hat{\nu_j}'(t), \\
\hat{\mu}_n''(t) & = & \sum_{j=1}^{n}\prod_{\ds i=1\atop \ds i\neq j}^{n}\hat{\nu_i}(t)\hat{\nu_j}''(t)+\sum_{j=1}^{n}\sum_{\ds k=1\atop \ds k\neq j}^{n}\prod_{\ds i=1\atop \ds i\neq j,k}^{n}\hat{\nu_{i}}(t)\hat{\nu_{j}}'(t)\hat{\nu_k}'(t).
\end{eqnarray*}
These imply that \\
\begin{eqnarray*}
|\hat{\mu}_{n}''(t)| & \leq & \sum_{j=1}^{n}a_j e^{-(n-1)Ct^2}+\sum_{j=1}^{n}a_j\sum_{\ds k=1\atop \ds k\neq j}^{n}a_k e^{-(n-2)Ct^2}|t|^2  \\
 & \leq & 4\pi^2\phi(n)e^{-(n-1)Ct^2}+16\pi^4 \phi(n)^2e^{-(n-2)Ct^2}|t|^2 ,
 \end{eqnarray*}
so that \\
\begin{eqnarray*}
\int_{-1/2}^{1/2}|\hat{\mu}_n''(t)||t|\,\mbox{d}t & \leq & 4\pi^2\phi(n)\int_{-1/2}^{1/2}e^{-(n-1)Ct^2}|t|\,\mbox{d}t+16\pi^4\phi(n)^2\int_{-1/2}^{1/2}e^{-(n-2)Ct^2}|t|^3\,\mbox{d}t\\
& \leq & I_1+I_2 .
\end{eqnarray*}
\begin{eqnarray*}
I_1 & = & 4\pi^2\phi(n)\int_{-1/2}^{1/2}e^{-(n-1)Ct^2}|t|\,\mbox{d}t=\\
 & = & 8\pi^2\phi(n)\int_{0}^{1/2}e^{-(n-1)Ct^2}t\,\mbox{d}t\\
 & = & 8\pi^2\phi(n)\left[\frac{e^{-(n-1)Ct^2}}{-2(n-1)C}\right]_0^{1/2}\\
 & = & 8\pi^2\phi(n)\left(\frac{e^{\frac{-(n-1)C}{4}}}{-2(n-1)C}+\frac{1}{2(n-1)C}\right)\\
 & = & 4\pi^2\frac{\phi(n)}{C(n-1)}\left(1-e^{-\frac{(n-1)C}{4}}\right) ,
 \end{eqnarray*}  
and\\
\begin{eqnarray*}
I_2 & = & 16\pi^4\phi(n)^2\int_{-1/2}^{1/2}e^{-(n-2)Ct^2}|t|^3\,\mbox{d}t \\
 & = & 32\pi^4\phi(n)^2\int_{0}^{1/2}e^{-(n-2)Ct^2}t^3\,\mbox{d}t\\
& = & 16\pi^4\phi(n)^2\int_0^{1/4}e^{-(n-2)Cu}u\,\mbox{d}u
\end{eqnarray*}
\begin{eqnarray*}
& = & 16\pi^4\phi(n)^2\left(\left.-\frac{ue^{-(n-2)Cu}}{(n-2)C}\right|_0^{1/4}+\frac{1}{(n-2)C}\int_0^{1/4}e^{-(n-2)Cu}\,\mbox{d}u\right)\\
& = & 16\pi^4\phi(n)^2\left(-\frac{e^{-\frac{(n-2)C}{4}}}{4(n-2)C}-\frac{1}{(n-2)^2C^2}\left.e^{-(n-2)Cu}\right|_0^{1/4}\right)\\
& = & 16\pi^4\phi(n)^2\left(-\frac{e^{-\frac{(n-2)C}{4}}}{4(n-2)C}-\frac{1}{(n-2)^2C^2}(e^{-\frac{(n-2)C}{4}}-1)\right)\\
& = & 16\pi^4 \left(-\frac{1}{4C}\left(\frac{\phi(n)}{n-2}\right)^2 e^{-\frac{(n-2)C}{4}}(n-2)-\frac{1}{C^2}\left(\frac{\phi(n)}{n-2}\right)^2(e^{-\frac{(n-2)C}{4}}-1)\right)
\end{eqnarray*}
Both integrals $\ds I_1$ and $\ds I_2$ are bounded, given that $\ds \phi(n)=O(n)$ .\\
Hence,
\begin{eqnarray*}
\sup_{n}\int_{-1/2}^{1/2}|\hat{\mu}''_{n}(t)||t|\,\mbox{d}t<\infty .
\end{eqnarray*}
\end{proof}

\newtheorem{l2}[thm2]{Lemma}
\begin{l2}[\cite{petrov}]
\label{vp}
Let $\ds f(t)$ be a characteristic function of a random variable $X$, then for all real numbers $t$, 
$\ds 1-|f(2t)|^2\leq 4(1-|f(t)|^2)$.
\end{l2}

\noindent This Lemma helps us prove the following result, which is a modification
of a Theorem found in~\cite{petrov}.

\newtheorem{l3}[thm2]{Lemma}
\begin{l3}
\label{vp2}
If $\ds |\hat{\mu}(t)|\leq c<1$ for $\ds \frac{1}{2}>|t|\geq b$
and for some $\ds b$ such that $\ds |b|<\frac{1}{4}$, then
$\ds |\hat{\mu}(t)|\leq 1-\frac{1-c^2}{8b^2}t^2$ for $\ds |t|\leq b$.
\end{l3}

\begin{proof}
For $\ds t=0$ the claim is obvious. Choose $\ds t$ such that $\ds |t|<b$. We can
find $n$ such that $\ds 2^{-n}b\leq |t|<2^{-n+1}b$, then $\ds b\leq 2^n|t|<2b$.
Hence $\ds |\hat{\mu}(2^n t)|\leq c$. Lemma~\ref{vp} implies that by induction
$\ds 1-|f(2^n t)|^2\leq 4^n(1-|f(t)|^2)$ holds for all $\ds t$ and any characteristic function $f$. 
Using the fact that $\ds \hat{\mu}(t)=\overline{f(2\pi t)}$ for $\ds -1/2\leq t<1/2$, we have that
\begin{eqnarray*}
1-|\hat{\mu}(2^nt)|^2 & = & 1-|\overline{f(2^n 2\pi t)}|^2 \\
& \leq & 4^n(1-|f(2\pi t)|^2)\\
& = & 4^n(1-|\hat{\mu}(t)|^2)\: , 
\end{eqnarray*}
which implies that
\begin{eqnarray*}
1-|\hat{\mu}(t)|^2 & \geq & \frac{1}{4^n}(1-|\hat{\mu}(2^n t)|^2) \\
& \geq & \frac{1}{4^n} (1-c^2)\geq \frac{1-c^2}{4b^2}t^2\: .
\end{eqnarray*}
$\ds |\hat{\mu}(t)| \leq 1-\frac{1-c^2}{8b^2}t^2$ for $\ds |t|<b$ follows.
\end{proof}


\newtheorem{corvp}[thm2]{Lemma}
\begin{corvp}
\label{corvp}
If $\ds\mu$ is a strictly aperiodic probability measure on $\ds\mathbb{Z}$ and $\ds\hat{\mu}(t)$ denotes the Fourier transform of $\ds\mu$ for $\ds t\in(-1/2,1/2]$, then there exist positive constants  $\ds c<1$ and $d$ such that $$|\hat{\mu}(t)|\leq 1-\frac{1-c^2}{8d^2}t^2\:\mbox{for}\: |t|\leq d$$ which implies that there exists $\ds C>0$ such that
$$ |\hat{\mu}(t)|\leq e^{-Ct^2}\:,\forall\:t\in[-1/2,1/2)\: .$$
\end{corvp}

\noindent The third condition of Theorem~\ref{pos} replaces the condition of strict aperiodicity
in the case when all of the $\ds\nu_i$'s are the same measure, i.e $\ds\nu_i=\nu$.

\newtheorem{newlemma}[thm2]{Lemma}
\begin{newlemma}
\label{lemma25}
Let $\ds \{\nu_n\}$ be a sequence of probability measures on $\ds\mathbb{Z}$ .\\
The following are equivalent \\
\begin{enumerate}
\item
$\ds\forall \delta>0$ 
\begin{eqnarray*}
\overline{\lim_{n\rightarrow\infty}}\sup_{|t|>\delta}|\hat{\nu}_n(t)|<1\:\mbox{(asymptotically strictly aperiodic)}
\end{eqnarray*}
\item
There exist $\ds C$ and $\ds N_0$ such that
\begin{eqnarray*}
|\hat{\nu}_n(t)|\leq e^{-Ct^2} \quad,\: \forall n>N_0
\end{eqnarray*}
\end{enumerate}
\end{newlemma}

\begin{proof}
$\ds (2)\Rightarrow (1)$ is obvious.\\
To show that $\ds (1)\Rightarrow (2)$ , since $\ds\forall\delta>0$ \,  $\ds\overline{\lim_{n\rightarrow\infty}}\sup_{\ds|t|>\delta}|\hat{\nu}_n(t)|<1$, \\
given $\ds\epsilon>0$ we can choose $\ds\delta>0$ and $\ds N\in\mathbb{Z}$ such that $\ds\sup_{|t|>\delta}|\hat{\nu}_n(t)|<1-\epsilon$, 
$\ds\forall n>N$. By Lemma~\ref{vp}, 
\begin{eqnarray*}
|\hat{\nu}_n(t)|& \leq & 1-Kt^2 
\end{eqnarray*}
for some constant $K$, $\ds n\geq N$ and $\ds |t|<\delta$ .\\
So that there exists a constant $C$ such that 
\begin{eqnarray*} 
|\hat{\nu}_n(t)|& \leq & e^{-Ct^2} 
\end{eqnarray*}
for all $\ds t\in [-1/2,1/2)$, $\ds \forall n\geq N$ .    
\end{proof}

\subsection{Dense set and Almost everywhere convergence in $\ds\mbox{L}^1(X)$}

\newtheorem{lmalmost}{Lemma}[subsection]
\begin{lmalmost}
\label{almost}
Let $\ds\mu_n$ be a sequence of probability measures on $\ds\mathbb{Z}$ such that
\begin{enumerate}
\item there is $\ds 0<\alpha\leq 1$ and $\ds C>0$ such that for $\ds n\geq 1$\\
\begin{eqnarray*}
|\mu_n(x+y)-\mu_n(x)|\leq C\frac{|y|^{\alpha}}{|x|^{1+\alpha}} & x,y\in\mathbb{Z} & 2|y|\leq |x|\: ,
\end{eqnarray*}
\item $\ds\hat{\mu}_n(t)\xrightarrow{n\rightarrow\infty}0$ for a.e $\ds t\in[-1/2,1/2)$\: .
\end{enumerate}
then\\
\begin{eqnarray*}
\|\mu_n-\mu_n\ast\delta_1\|_1\xrightarrow{n\rightarrow \infty}0\: .
\end{eqnarray*}
\end{lmalmost}

\begin{proof}
Note that by the first assumption 
\begin{eqnarray*}
|\mu_n(k)-\mu_n\ast\delta_1(k)| & =  & |\mu_n(k-1+1)-\mu_n(k-1)|\\
&\leq & C\frac{1}{(k-1)^{1+\alpha}} \quad \mbox{for}\:2<|k-1|\: .
\end{eqnarray*}
This implies that the sequence $\ds |\mu_n(k)-\mu_n\ast\delta_1(k)|$ is bounded by a summable function. By Lebesgue's dominated convergence theorem the condition
\begin{eqnarray*}
\|\mu_n-\mu_n\ast\delta_1\|_1\xrightarrow{n\rightarrow \infty}0
\end{eqnarray*}
holds if we show that 
\begin{eqnarray*}
|\mu_n(k)-\mu_n(k-1)|\xrightarrow{n\rightarrow\infty}0\quad\forall k\: .
\end{eqnarray*}
Indeed, observe that
\begin{eqnarray*}
|\mu_n(k)-\mu_n(k-1)| & = & \left|\int_{-1/2}^{1/2}\hat{\mu}_n(t)(e^{-2\pi ikt}-e^{-2\pi i(k-1)t})\,\mbox{d}t\right|\\
& \leq & \int_{-1/2}^{1/2}|\hat{\mu}_n(t)||e^{-2\pi ikt}||1-e^{2\pi it}|\,\mbox{d}t\rightarrow0\:\mbox{as}\: n\rightarrow\infty\: .
\end{eqnarray*}
by $(2)$ and the Bounded Convergence Theorem.
\end{proof}

\newtheorem{almostevery}[lmalmost]{Theorem}
\begin{almostevery}
\label{almostevery}
Let $\ds (\nu_{n})$ be a sequence of strictly aperiodic probability measures on $\ds\mathbb{Z}$ such that
\begin{enumerate}
\item $\ds E(\nu_n)=0$\, , \, $\ds\forall n$
\item
$\ds \phi(n)=\sum_{i=1}^{n} m_2(\nu_i)=O(n) $
\item 
There exist a constant $C$ and an integer $\ds N_0>0$, such that $\ds |\hat{\nu}_{n}(t)|\leq e^{-Ct^2},$ $\ds\forall n>N_0$ and $\ds t\in[-1/2,1/2)$\: . 
\end{enumerate}
The sequence of operators $\ds\{\mu_nf\}$ converges almost everywhere in $\ds\mbox{L}^1(X)$. 
\end{almostevery}
\begin{proof}
Since the maximal operator has been shown to be of weak-type $(1,1)$ (Theorem~\ref{pos}), it is enough to show that convergence
holds on the dense class $\ds\{f+g-g\circ\tau: f\circ\tau=f,\:g\in\mbox{L}_{\infty}\}$.
Clearly, $\ds\mu_n f$ converges almost everywhere for $\ds\tau-$invariant functions $f$.
Then, to show that 
$\ds(\mu_n g-\mu_n(g\circ\tau))$ converges almost everywhere for $\ds g\in\mbox{L}_{\infty}$
it is enough to show that 
\begin{eqnarray*}
\|\mu_ng-\mu_n(g\circ\tau)\|_{\infty}\xrightarrow{n\rightarrow\infty}0
\end{eqnarray*}
But 
\begin{eqnarray*}
\|\mu_ng-\mu_n(g\circ\tau)\|_{\infty} & \leq & \|\mu_ng-(\mu_n\ast\delta_1) g\|_{\infty}  \\
& \leq & \|\mu_n-\mu_n\ast\delta_1\|_1\|g\|_{\infty}
\end{eqnarray*}
So that it is enough to show the following
\begin{eqnarray*}
\|\mu_n-\mu_n\ast\delta_1\|_1\xrightarrow{n\rightarrow\infty}0
\end{eqnarray*}  
which holds according to Lemma ~\ref{almost}.
\end{proof}


\theoremstyle{definition}

\section{Collections with uniformly bounded second moments}
\theoremstyle{plain}

\newtheorem{lems}{Lemma}[section]
\begin{lems}
\label{lems}
Let $\ds A\subseteq \mathbb{C}^4$ be the set 
$$A=\{(a_1,a_2,z_1,z_2)\, :\, a_1+a_2=1\, , \, a_1,a_2\geq 0\, , \, |z_1|=|z_2|=1\}$$ and $\ds S(\delta,\eta)\subseteq A$ be the set 
$$S(\delta,\eta)=\{(a_1,a_2,z_1,z_2)\,:\, a_1,a_2\geq \delta\, \mbox{and}\, |z_1-z_2|\geq\eta\}\, ,0< \delta,0<\eta\}\, ,$$
then there exists $\ds \rho=\rho(\delta,\eta)<1$ such that $\ds\forall(a_1,a_2,z_1,z_2)\in S(\delta,\eta)$ \\
$\ds |a_1z_1+a_2z_2|\leq\rho$ holds.
\end{lems}
\begin{proof}
By the triangle inequality for points in $A$ $\ds |a_1z_1+a_2z_2|=1$ if and only if $\ds a_1z_1=\lambda a_2z_2$ for $\ds \lambda\geq 0$, which implies that $\ds (a_1,a_2,z_1,z_2)\in A$. Therefore 
$\ds F(a_1,a_2,z_1,z_2)=a_1z_1+a_2z_2$ has modulus $1$ on $A$ only on the set
$\ds R=\{(a_1,a_2,z_1,z_2)\, , \, a_1=a_2\, ,\,z_1=z_2\}$. Observe that the points 
in $\ds S(\delta,\eta)$ are bounded away from $R$. Since $\ds S(\delta,\eta)$
is a compact subset of $A$ and $F$ is continuous on $A$ the claim follows.
\end{proof}

\newtheorem{lemrho}[lems]{Lemma}
\begin{lemrho}
\label{lemrho}
Let $\ds\nu$ be a probability measue on $\mathbb{Z}$ with $\ds m_1(\nu)\leq a$
and \\
$\ds \sup_{\beta,r\in\mathbb{Z}}\nu(\beta\mathbb{Z}+r)\leq \rho<1$. Suppose $\ds\frac{l}{s}$ is a rational number in $\ds (-1/2,1/2]$ with $\ds |s|\leq M$ and $\ds |l|\leq \left\lfloor{\frac{|s|}{2}}\right\rfloor$. Then there exists $\ds 0\leq\sigma=\sigma(a,\rho)<1$ such that $\ds\left|\hat{\nu}\left(\frac{l}{s}\right)\right|\leq\sigma$.
\end{lemrho}
\begin{proof}
Let $\ds |s|\leq M$. We have for $\ds |l|\leq\left\lfloor\frac{s}{2}\right\rfloor$, 
$\ds \hat{\nu}\left(\frac{l}{s}\right)=\sum_{m\in\mathbb{Z}}\nu(m)e^{2\pi im(l/s)}$. Write $\ds d=\gcd(l,s)$, then $\ds l=d\alpha$, $\ds s=d\beta$, and $\ds m=\gamma\beta+r$ for some $\ds 0\leq r<\beta$. Then $$\hat{\nu}\left(\frac{l}{s}\right)=\sum_{r=0}^{\beta-1}\nu(\beta\mathbb{Z}+r)e^{2\pi ir(\alpha/\beta)}\, .$$ By assumption there exist two cosets $\ds \beta\mathbb{Z}+r_1$, $\ds\beta\mathbb{Z}+r_2$ and a value $\delta$, which depends only on $M$ and $\rho$ such that $\ds \nu(\beta\mathbb{Z}+r_1)$, $\ds\nu(\beta\mathbb{Z}+r_2)\geq\delta$. Therefore,
\begin{eqnarray*}
\hat{\nu}\left(\frac{l}{s}\right) & = &\nu(\beta\mathbb{Z}+r_1)e^{2\pi ir_1(\alpha/\beta)}+\nu(\beta\mathbb{Z}+r_2)e^{2\pi ir_2(\alpha/\beta)} \\
& + & \sum_{m  \notin \beta\mathbb{Z}+r_1\cup\beta\mathbb{Z}+r_2),}\nu(m)e^{2\pi im(\alpha/\beta)}\, .
\end{eqnarray*}
Also since $\ds\gcd(\alpha,\beta)=1$ 
$$ |e^{2\pi ir_1(\alpha/\beta)}-e^{2\pi ir_2(\alpha/\beta)}|=|1-e^{2\pi i(r_2-r_1)(\alpha/\beta)}|\geq \eta>0\, ,$$
where $\eta$ depends on $M$ and $\rho$ since $\ds |\beta|\leq|s|\leq M$.
Therefore by Lemma~\ref{lems} there exists a $\ds 0\leq\sigma'=\sigma'(M,\rho)<1$ such that
$$|\nu(\beta\mathbb{Z}+r_1)e^{2\pi ir_1(\alpha/\beta)}+\nu(\beta\mathbb{Z}+r_2)e^{2\pi 
ir_2(\alpha/\beta)}|\leq\sigma'\left(\nu(\beta\mathbb{Z}+r_1)+
\nu(\beta\mathbb{Z}+r_2)\right)\, .$$
It follows that there exists $\ds 0\leq\sigma=\sigma(M,\rho)<1$ such that $\ds|\hat{\nu}(l/s)|\leq\sigma$.
\end{proof}

\newtheorem{thmchris}[lems]{Theorem}

\begin{thmchris}
\label{thmchris}
Let $\ds\nu$ be a probability measue on $\mathbb{Z}$ with $\ds m_1(\nu)\leq a$
and $\ds \sup_{\beta,r\in\mathbb{Z}}\nu(\beta\mathbb{Z}+r)\leq \rho<1$. Then there exists a $c=c(a,\rho)$ such that $\ds |\hat{\nu}(t)|\leq e^{-ct^2}$. 
\end{thmchris}
\begin{proof}
By hypothesis and using Chebyshev's inequality there exist $\ds\delta=\delta(\rho,a)$, $\ds M=M(a)$, and integers $k$, $j$
such that $\ds |k|$, $|j|\leq M$ and $\ds \nu(k)$, $\ds \nu(j)\geq\delta$. Let
$\ds s=k-j$, and consider the points $\ds \left\{\frac{p}{s}\, :\, p=0,\pm 1,\ldots,\pm \left\lfloor\frac{|s|}{2}\right\rfloor\right\}$. By Lemma~\ref{lemrho} and 
the mean value theorem for $\ds p=\pm 1,\ldots,\pm\frac{|s|}{2}$ there exists an $\ds\epsilon=\epsilon(a)$ such that $\ds \forall t\in\left(\frac{p}{s}-\epsilon,\frac{p}{s}+\epsilon\right)$ we have 
$\ds |\hat{\nu}(t)|\leq\sigma+\frac{1-\sigma}{2}$ where $\sigma$ is the value 
in Lemma~\ref{lemrho} . Let $\ds I_{p}=\left(\frac{p}{s}-\epsilon,\frac{p}{s}+\epsilon\right)$ where
$\ds p=0,\pm 1,\ldots,\pm\left\lfloor\frac{|s|}{2}\right\rfloor$, and $t_0$ a point
in the complement of $\ds S=\cup_{p}I_p$. We have 
\begin{eqnarray*}
\hat{\nu}(t_0) & = &\nu(k)e^{2\pi ikt_0}+\nu(j)e^{2\pi ijt_0}\\
& + & \sum_{m\neq k,j}\nu(m)e^{2\pi imt_0}\, .
\end{eqnarray*}
Now, $\ds |e^{2\pi ikt_0}-e^{2\pi ijt_0}|=|1-e^{2\pi ist_0}|$ and this is greater than a value $\eta>0$, which depends only on $s$ and $\epsilon$
which depends only on $m_1(\nu)$ which is bounded by $a$. Thus by Lemma~\ref{lems}
$$|\nu(k)e^{2\pi ikt_0}+\nu(j)e^{2\pi ijt_0}|\leq\sigma'(\nu(k)+\nu(j))$$
and therefore $\ds |\hat{\nu}(t_0)|\leq\sigma''<1$ for some value $\sigma''=\sigma''(\rho,a)$. We therefore have for $|t|\geq \epsilon$ a value
$\ds \sigma'''=\max(\sigma,\sigma'')<1$ dependent on $\rho$ and $a$ only, such that $\ds|\hat{\nu}(t)|\leq \sigma''$. By Lemma \ref{vp} there exists a $c'$ such that
$\ds |\hat{\nu}(t)|\leq 1-c't^2<1$ for $\ds 0<|t|<\epsilon$. The conclusion follows by choosing a value $c$ small enough so that 
$\ds |\hat{\nu}(t)|\leq e^{-ct^2}$ $\ds\forall t\in(-1/2,1/2]$.
\end{proof}

\noindent Combining Theorems~\ref{almostevery} and \ref{thmchris} we get,

\newtheorem{thmcor}[lems]{Theorem}
\begin{thmcor}
If $\nu_n$ is a sequence of probability measures on $\ds \mathbb{Z}$ such that
for all $n$ 
\begin{enumerate}
\item $\ds E(\nu_n)=0$
\item $\ds m_1(\nu_n)\leq a$
\item $\ds \sup_n\sup_{\alpha,\beta}\nu_n(\beta\mathbb{Z}+\alpha)\leq\rho<1$
\item $\ds \phi(n)=\sum_{i=1}^{n}m_2(\nu_i)=O(n)$
\end{enumerate}
then $\ds\mu_nf(x)$ converges a.e. $\forall f\in\mbox{L}^1(X)$.
\end{thmcor}
\theoremstyle{definition}
\newtheorem{remnew}[lems]{Remark}
\begin{remnew}
Let 
$$\nu_n(k)=\left\{
\begin{array}{ll}
\ds\frac{1-a_n}{2} & k=\pm 1\\
a_n & k=0 \\
0 & \mbox{otherwise}
\end{array}\right.$$
where $1>a_n>0$ and $a_n\rightarrow 0$ fast enough so that $\prod_{n=1}^{\infty} a_n>0$. Then, using an argument similar to that in \cite{belrjblatt} one may show that the sequence $\mu_n f$ does not converge a.e for some $f\in \mbox{L}^{\infty}$. Of course the sequence $\nu_n(k)$ does not satisfy the condition $\ds\sup_n\sup_{\alpha,\beta}\nu_n(\beta\mathbb{Z}+\alpha)\leq\rho$ while it does satisfy the condition $\ds m_1(\nu_n)\leq a$.
\end{remnew}

\section{The Strong Sweeping Out Property}
\theoremstyle{plain}
\subsection{Introduction}

In this section $\ds (X,\mathcal{B},m,\tau)$ and $\ds\tau $ are as previously. Hereby, we discuss cases where the operators $\ds\mu_n f(x)=\sum_{k\in\mathbb{Z}}\mu_n(k)f(\tau^k x)$ fail to
converge, where as before $\ds \mu_n=\nu_1\ast\cdots\ast\nu_n$. The case where $\ds \mu_n$ is given by the convolution powers of a single probability measure $\mu$ on $\ds\mathbb{Z}$
, i.e $\ds \mu_n=\mu^n$, has been studied. In the event of convolution
powers, the probability measure $\ds\mu$ given by $\ds\mu=\frac{1}{2}(\delta_{0}+\delta_{1})$ is the prototype 
of bad behavior for the resulting sequence of operators $\ds (\mu^n f)(x)$. By using the Central Limit Theorem
it was shown in ~\cite{belrjblatt} that the bad behavior of this prototype is typical at least if $\mu$ has
$\ds m_2(\mu)<\infty$ and $\ds E(\mu)\neq 0$(~\cite{belrjblatt}). In ~\cite{loserttwo}, this result was extended to probability measures with $\ds E(\mu)=0$ and $\ds m_{p}(\mu)<\infty$
for $\ds p>1$.


\theoremstyle{Definition}

\newtheorem{stswout}{Definition}[subsection]
\begin{stswout}
The sequence of measures $\ds\mu_n$ is said to have the \textbf{strong sweeping out} property,
if given $\ds\epsilon>0$, there is a set $\ds B\in\mathcal{B}$ with $\ds m(B)<\epsilon$ such that
$$\limsup_{n}\mu_n\chi_{B}(x)=1\:\mbox{a.e.}\:\liminf_{n}\mu_n\chi_{B}(x)=0\:\mbox{a.e.}$$
\end{stswout}



\theoremstyle{plain}
\noindent We will use the following in our constructions.
\newtheorem{jr}[stswout]{Proposition}
\begin{jr}[\cite{joe}]
\label{joeneg}
For any sequence of probability measures $\ds\mu_N$ on $\ds\mathbb{Z}$ that are dissipative, i.e. $\ds\lim_{N\rightarrow\infty}\mu_N(k)=0$
for all $\ds k\in\mathbb{Z}$, if there exists $\ds b>0$ and a dense subset $\ds D\subset \{\gamma : |\gamma|=1\}$ with 
$\ds\liminf_{N\rightarrow\infty}|\hat{\mu}_{N}(\gamma)|\geq b$ for all $\ds\gamma\in D$, then for any ergodic dynamical system 
$\ds(X,\mathcal{B},m,\tau)$ the sequence $\mu_n$ is strong sweeping out.
\end{jr}

\subsection{Strong sweeping out with convolution measures}

\newtheorem{chrislemma2}{Theorem}[subsection]
\begin{chrislemma2}
\label{chrislemma2}
If $\ds \nu_n=a_n\delta_{x_n}+(1-a_n)\gamma_n$, where $\ds \gamma_n$ is a probability measure,
$\ds\sum x_n$ either $\ds\rightarrow\infty$ or $\ds\rightarrow -\infty$ and $\ds\sum_{n}(1-a_n)<\infty$, then
$\ds \{\mu_n=\nu_1\ast\cdots\ast\nu_n\}$ is a dissipative sequence.
\end{chrislemma2}

\begin{proof}Without loss of generality, assume that $\ds\sum x_n\rightarrow \infty$.
Suppose $\ds\nu_n=a_n\delta_{x_n}+(1-a_n)\gamma_n$ as above, then we have $\ds \sum P(Z_n\neq x_n)\leq\sum(1-a_n)<\infty$,
where $\ds Z_n$ is a sequence of independent random variables having distribution $\ds \nu_n$. By the Borel-Cantelli Lemma
$\ds P(Z_n\neq x_n\:\mbox{i.o})=0$. Let $\ds\omega\in(Z_n\neq x_n\:\mbox{i.o})^{c}$. Then
\begin{eqnarray*}
S_N(\omega)=\sum_{m=1}^{N}Z_m(\omega) & = & \sum_{Z_m(\omega)\neq x_m}z_n +\sum_{Z_m(\omega)=x_m}x_m\\
& \geq & -c(\omega) + \sum_{Z_m(\omega)=x_m}\!x_m\rightarrow\infty\:\mbox{as}\: N\rightarrow\infty
\end{eqnarray*} 
as $\ds c(\omega)$ is a constant depending on $\omega$. Hence $\ds S_{N}(\omega)\rightarrow\infty$ with probability $1$.
Therefore, when $k$ is fixed, $\ds P(S_N=k)\rightarrow 0$. Indeed, since $\ds P(\bigcup_{N=1}^{\infty}(S_m(\omega)>k\:\forall m
\geq N))=1$ and the sequence of sets is increasing, we have $\ds P(S_m(\omega)>k\:\forall m\geq N)\rightarrow 1$. But 
$\ds P(S_{N}>k)\geq P(S_m>k\:\forall m\geq N)$ so $\ds P(S_N(\omega)=k)\leq 1-P(S_N(\omega)>k)\rightarrow 0$.
Hence, $\ds \lim_{n\rightarrow\infty}\mu_n(k)=\lim_{n\rightarrow\infty}(\nu_1\ast\cdots\ast\nu_n)(k)=0$ and $\ds \{\mu_n\}$ is a dissipative sequence.
\end{proof}

\newtheorem{gammacor}[chrislemma2]{Corollary}
\begin{gammacor}
\label{annneg}
Let $\ds \nu_n=a_n\delta_{x_n}+(1-a_n)\gamma_n$, where $\ds \gamma_n$ is a probability measure on $\ds\mathbb{Z}$, such that
$\ds x_n\in\mathbb{Z}$, $\ds\sum(1-a_n)<\infty$, $\ds |x_n|\geq 1$ and $\ds\sum x_n\rightarrow\infty\:\mbox{or}\: -\infty$, 
then for any ergodic dynamical system $\ds(X,\mathcal{B},m,\tau)$ 
then the sequence $\mu_n=\nu_1\ast\nu_2\ast\cdots\ast\nu_n$ is strong sweeping out.
\end{gammacor}

\begin{proof}
Theorem~\ref{chrislemma2} implies that the sequence $\ds\mu_n=\nu_1\ast\cdots\ast\nu_n$ is dissipative.
Note that for $\ds t\in[-1/2,1/2)$ we have 
\begin{eqnarray*}
|\hat{\mu_n}(t)| & = & \prod_{l=1}^{n}|\hat{\nu}_l(t)|\\
                 & = & \prod_{l=1}^{n}|a_l e^{-2\pi ix_l t}+(1-a_l)\hat{\gamma}_l(t)|\\
                 & \geq & \prod_{l=1}^{n}\left||a_l|-(1-a_l)|\hat{\gamma_l}(t)|\right|\\
                 & \geq & \prod_{l=1}^{n}\left||a_l|-(1-a_l)\right|=\prod_{l=1}^{n}(2a_l-1)\\
                 & = & \prod_{l=1}^n a_l(2-\frac{1}{a_l})\geq c\prod_{l=1}^N a_l\geq cc'>0\: .
\end{eqnarray*}
The result follows by Proposition~\ref{joeneg}. Note that we have used the fact that for $\ds a_l>0$
such that $\ds \sum(1-a_l)<\infty$ implies that $\ds\prod a_l$ converges to a nonzero value.
\end{proof}

\theoremstyle{plain}
\newtheorem{chrislemma}[chrislemma2]{Lemma}
\begin{chrislemma}
Let $\ds \nu_n=a_n\delta_{x_n}+(1-a_n)\gamma_n$, where $\ds \gamma_n$ is a probability measure,
$\ds E(\nu_n)=0$, $\ds |x_n|\geq c$ and $\ds a_n\geq d$ for some constants $c$ and $d$, then $\ds m_2(\nu_n)\geq\frac{\alpha}{1-a_n}$ where
$\alpha=dc^2$ .
\end{chrislemma}

\begin{proof}
Since $\ds E(\nu_n)=a_n x_n+(1-a_n)E(\gamma_n)=0$, $\ds \frac{a_n x_n}{a_n-1}=E(\gamma_n)$. Therefore 
\begin{eqnarray*}
 m_2(\nu_n)=a_n x_n^2+(1-a_n)m_2(\gamma_n) & \geq & a_n x_n^2+|E(\gamma_n)|^2(1-a_n) \\
 & = & a_n x_n^2+\frac{a_n^2 x_n^2}{1-a_n}\\
 & \geq &  \frac{\alpha}{1-a_n}
\end{eqnarray*}
This provides a lower bound on the second moment, i.e $\ds m_2(\nu_n)\geq \frac{\alpha}{1-a_n}$.
If in addition $\ds \sum(1-a_n)<\infty$, once we allow $\ds\prod a_n\geq c>0$, the second moments $\ds m_2(\nu_n)$
can not grow arbitrarily slowly. 
\end{proof}

\theoremstyle{definition}

\newtheorem*{egneg}{Example}

\begin{egneg}
Let $\ds a_n$ be a sequence such that $\ds\sum(1-a_n)<\infty$. Let $\ds b_n=\left[\frac{1}{1-a_n}\right]$, where $\ds [x]$ denotes the integer part of the number $x$.
Consider the measures given by\\
$\ds \nu_n(k)=\left\{
\begin{array}{lcl}
\ds\frac{1+2b_n}{3+2b_n} & , & k=1 \\ \vspace{.002in} 
\ds\frac{1}{3+2b_n}& , & k=-b_n\\ \vspace{.002in}
\ds\frac{1}{3+2b_n} & , & k=-b_n-1
\end{array}\right.$ \\ \vspace{.001in} \\
These measures satisfy the assumptions of
Theorem~\ref{annneg}. As such, the sequence $\ds\mu_n f=(\nu_1\ast\cdots\ast\nu_n)f$ is strong sweeping out.
It is noteworthy that all the measures in this example satisfy additionally
the following property
\begin{eqnarray*} 
m_2(\nu_n) & = & \frac{2b_n^2+4b_n+2}{3+2b_n},
\end{eqnarray*}
which implies that the second moment grows like $\ds \frac{1}{1-a_n}$.
\noindent One might think of this sequence $\ds \nu_n$ as 
$$ \nu_n=a_n\delta_{1}+\frac{(1-a_n)}{2}(\delta_{-b_n}+\delta_{-b_n-1})=a_n\delta_{1}+(1-a_n)\gamma_n$$
where $\ds \gamma_n=1/2(\delta_{-b_n}+\delta_{-b_n-1})$.
\end{egneg}

\end{document}